\documentclass[11pt]{amsart}
\usepackage{graphicx,amssymb,amsmath,amsthm}
\usepackage{amsfonts}
\usepackage{enumerate}
\usepackage{dsfont}
\usepackage{cite}
\usepackage[colorlinks, citecolor=red]{hyperref}
\usepackage{mathrsfs}
\usepackage{epsfig}
\usepackage{lscape}
\usepackage{subfigure}
\usepackage{epstopdf}
\usepackage{caption}
\usepackage{algorithm}
\usepackage{algpseudocode}
\usepackage{multirow}
\usepackage{geometry}
\usepackage{paralist}
\usepackage{enumerate}
\usepackage{url}

\newtheorem{definition}{Definition}[section]

\newtheorem{theorem}[definition]{Theorem}
\newtheorem{lemma}[definition]{Lemma}

\newtheorem{remark}[definition]{Remark}

\usepackage{booktabs}
\usepackage{caption}
\usepackage{subfigure}
\usepackage{graphicx}
\usepackage{float}
\usepackage{color}

\date{}

\begin{document}
\baselineskip 18pt
\bibliographystyle{plain}
\title[A GRABP method for linear feasibility problems]
{A greedy randomized average block projection method for linear feasibility problems}

\author{Lin Zhu}
\address{School of Mathematics, Hunan University, Changsha 410082, China.}
\email{zhulin@hnu.edu.cn}

\author{Yuan Lei}
\address{School of Mathematics, Hunan University, Changsha 410082, China.}
\email{yleimath@hnu.edu.cn}

\author{Jiaxin Xie}
\address{LMIB of the Ministry of Education, School of Mathematical Sciences, Beihang University, Beijing, 100191, China. }
\email{xiejx@buaa.edu.cn}

\begin{abstract}
The randomized projection (RP) method is a simple iterative scheme for solving linear feasibility problems and has recently gained  popularity due to its speed and low memory requirement. This paper develops an accelerated variant of the standard RP method by using  two ingredients: the \emph{greedy probability} criterion and the \emph{average block} approach, and obtains a \emph{greedy randomized average block projection} (GRABP) method  for  solving large-scale systems of linear inequalities. We prove that this method  converges linearly in expectation under different choices of extrapolated stepsizes. Numerical experiments on both randomly generated and real-world  data  show the advantage of GRABP  over several state-of-the-art solvers, such as the randomized projection (RP) method, the sampling Kaczmarz Motzkin (SKM) method, the generalized SKM (GSKM) method, and the  Nesterov  acceleration of SKM method.
\end{abstract}

\let\thefootnote\relax
\maketitle

\section{Introduction}\label{sec1}

\subsection{Model and Notation}
We consider the problem of solving large-scale systems of linear inequalities
\begin{equation}\label{inequalities}
	Ax\leq b,
\end{equation}
where $A \in \mathbb{R}^{m\times n}$ and $b \in \mathbb{R}^{m}$.
We confine the scope of this work  to the regime of  $m\gg n$, where iterative methods are more competitive for such problems.  We denote the feasible region of \eqref{inequalities} by $S= \{x \in \mathbb{R}^{n}\mid Ax\leq b\}$.
Throughout this paper, we assume that the coefficient matrix $A$ has no zero rows and $S\neq\emptyset$.

For a given  matrix  $G$,  we use $\| G\|_2$, $\| G\|_F$, and $G^\dagger$  to denote the spectral norm, the Frobenius  norm, and the Moore-Penrose pseudoinverse, respectively.  We use $\sigma_{\min}(G)$ to denote  the smallest nonzero singular value of the matrix $G$. For an integer $m\geq 1$, let $[m]:=\{1,\ldots,m\}$.  For any vector $x\in\mathbb{R}^n$, we use $x_i,x^\top$, $\|x\|_2$, and $\|x\|_p$ to denote the $i$-th entry, the transpose, the Euclidean norm and the $p$-norm of $x$, respectively. For any $u\in\mathbb{R}$ and $v\in\mathbb{R}^n$, we define $(u)_+=\max\{0,u\}$ and $(v)_{+}=((v_1)_+,\ldots,(v_n)_+)^\top$. We refer to $\{\mathcal{I}_1,\mathcal{I}_2,\cdots,\mathcal{I}_t\}$ as a partition of $[m]$ if  $\mathcal{I}_i\bigcap \mathcal{I}_j=\emptyset$ for $i\neq j$ and $\bigcup^t_{i=1}\mathcal{I}_i=[m]$. For a given index set $\mathcal{I}_i$, we use $G_{\mathcal{I}_i,:}$ to denote the row submatrix of the matrix $G$ indexed by $\mathcal{I}_i$ and $u_{\mathcal{I}_i}$ denote the subvector of the vector $u$ with components listed in $\mathcal{I}_i$.
We use $P_S(u)$ to represent the orthogonal projection of $u$ onto the  feasible  region $S$.
For any random variables $\xi$ and $\zeta$, we use $\mathbb{E}[\xi]$ and $\mathbb{E}[\xi\lvert \zeta]$ to denote the expectation of $\xi$ and the conditional expectation of $\xi$ given $\zeta$, respectively.

\subsection{The randomized Kaczmarz method}
The  Kaczmarz method \cite{Kac37}, also known as \emph{  the algebraic reconstruction technique} (ART) \cite{herman1993algebraic,gordon1970algebraic},  is a widely used algorithm for solving the linear system $Ax=b$.
Starting from  $x^0\in\mathbb{R}^n$, the canonical Kaczmarz method constructs $x^{k+1}$ by
$$
x^{k+1}=x^k-\frac{\langle A_{i,:},x^k\rangle-b_i}{\|A_{i,:}\|^2_2}A_{i,:},
$$
where $ i $ is selected from $ [m ]$ cyclically.
In fact, the current iterate is projected orthogonally onto the selected hyperplane  $\{x\mid \langle A_{i,:},x\rangle=b_i\}$ at each iteration.
The iteration sequence $\{x^k\}_{k=0}^\infty$ converges to $x^0_*:=A^\dagger b+(I-A^\dagger A)x^0$. However, the rate of convergence is hard to obtain. In the seminal paper \cite{Str09}, Strohmer and Vershynin first analyzed the randomized variant of the Kaczmarz method (RK). Specifically, they proved that if the $i$-th row of $A$ is selected with probability proportional to $\|A_{i,:}\|_2^2$, then the method converges linearly in expectation.

 Leventhal and Lewis \cite{Lev10} extended the randomized Kaczmarz method to solve the linear feasibility problem \eqref{inequalities}.
At each iteration $k$, if the inequality is already satisfied for the selected row $i$, then set $x_{k+1}=x_{k}$. If the inequality is not satisfied, the previous iterate only projects onto the solution hyperplane $\{x\mid \langle A_{i,:},x\rangle=b_i\}$. The update rule for this algorithm is thus
\begin{equation}\label{pj-lf}
	x^{k+1}=x^k-\frac{(A_{i,:}x^k-b_i)_+}{\|A_{i,:}\|^2_2}(A_{i,:})^\top.
\end{equation}
One can see that $x^{k+1}$ in \eqref{pj-lf} is indeed the projection of $x^k$ onto the set $\{x\mid A_{i,:}x\leq b_i\}$.
Leventhal and Lewis \cite{Lev10} (Theorem 4.3) proved that such  \emph{randomized projection} (RP) method converges to a feasibility solution linearly in expectation.

Recently, by  combining the ideas of Kaczmarz and Motzkin methods \cite{agmon1954relaxation,motzkin1954relaxation}, Loera, Haddock, and Needell \cite{De17} proposed the \emph{sampling Kaczmarz-Motzkin} (SKM) method for solving   the linear feasibility problem \eqref{inequalities}. Later, Morshed, Islam, and Noor-E-Alam \cite{morshed2021sampling} developed a generalized framework, namely the \emph{generalized sampling Kaczmarz-Motzkin} (GSKM) method that
extends the SKM algorithm and proves the existence of a family of SKM-type methods. In addition, they also proposed a Nesterov-type acceleration scheme in the SKM method called
\emph{probably accelerated sampling Kaczmarz-Motzkin} (PASKM), which provides a
bridge between Nesterov-type acceleration of machine learning to sampling Kaczmarz
methods for solving linear feasibility problems.

\subsection{The greedy probability criterion}

The greedy probability criterion was originally proposed by Bai and Wu \cite{Bai18Gre} for effectively selecting the working row
from the matrix $A$, and a \emph{greedy randomized Kaczmarz} (GRK)  method which is faster than the RK method in terms of the number of iterations and computing time is introduced.
Indeed, at the $k$-th iteration,  GRK determines a subset $\mathcal{U}_k$ of $[m]$ such that the magnitude of the residual $\langle A_{i,:},x^{k}\rangle-b_i$ exceeds a threshold
i.e.,
$$
\mathcal{U}_k=\left\{ i_k \ \big| \ | \langle A_{i_k,:},x^{k}\rangle-b_{i_k}|^2 \geq \varepsilon_k\|Ax^{k}-b\|_2^2\|A_{i_k,:}\|_2^2\right\},
$$
where
$
\varepsilon_k=\frac{1}{2}\left(\frac{1}{\left\|Ax^{k}-b\right\|_2^2} \max\limits _{1 \leq i \leq m}\left\{\frac{\left|\langle A_{i,:},x^{k}\rangle-b_i\right|^2}{\left\|A_{i,:}\right\|_2^2}\right\}+\frac{1}{\|A\|_F^2}\right).
$
Then, a modified residual vector $\tilde{r}^{k}$ is defined by
$$
\tilde{r}_i^{k}= \begin{cases}\langle A_{i,:},x^{k}\rangle-b_i, & \text { if } i \in \mathcal{U}_k, \\ 0, & \text { otherwise}.\end{cases}
$$
GRK selects the index $i_k \in \mathcal{U}_k$ of the working row with  probability
$$
\operatorname{Pr}\left(\text{row }=i_k\right)=\frac{\left|\tilde{r}_{i_k}^k\right|^2}{\left\|\tilde{r}^{k}\right\|_2^2} .
$$
Finally, GRK orthogonally projects the current iterate $x^k$ onto the $i_k$-th hyperplane $\{x\mid \langle A_{i_k,:},x\rangle=b_i\}$ to obtain the next iterate $x^{k+1}$. By using the above greedy idea, small entries of the residual vector $Ax^k-b$ may not be selected, which guarantees the progress of each iteration of GRK and a faster convergence rate of GRK may be expected than that of RK.
The idea of greed applied in the literature \cite{Bai18Gre} has wide applications and has been used in many works, see for example the literatures \cite{morshed2020stochastic,bai2021greedy,Gow19,yuan2022adaptively} and  the references therein.

\subsection{The block Kaczmarz method}
The block Kaczmarz method first partitions the rows $[m]$ into $t$ blocks, denoted $\mathcal{I}_1,\ldots,\mathcal{I}_t$. Instead of selecting one row per-iteration as done with the simple Kaczmarz method, the block Kaczmarz algorithm chooses a block uniformly at random at each iteration.
Needell and Tropp \cite{needell2014paved} proposed a \emph{randomized block Kaczmarz} (RBK), where at each iteration, the previous iterate $x_k$ is projected onto the solution space to $A_{\mathcal{I}_{i_k},:}x = b_{\mathcal{I}_{i_k}}$. However, each iterate of this RBK method needs to apply the pseudoinverse of the chosen submatrix to a vector and it is expensive.

Recently, Necoara \cite{Nec19} developed a \emph{randomized average
	block Kaczmarz} (RABK)
algorithm for linear systems which takes a convex combination of several
RK updates as a new direction with some stepsize.
Assuming that the subset $\mathcal{I}_{i_k}$ has been selected at the $k$-th iteration, RABK generates the $k$-th estimate $x^{k+1}$ via
\begin{equation}\label{rabk-g}
	x^{k+1}=x^{k}-\alpha_k\left(\sum_{j \in \mathcal{I}_{i_k}} \omega_j^k \frac{A_{j,:} x^{k}-b_j}{\left\|A_{j,:}\right\|_2^2}\left(A_{j,:}\right)^{\top}\right),
\end{equation}
where the weights $\omega_j^k \in[0,1]$ such that $\sum_{j \in \mathcal{I}_{i_k} } \omega_j^k=1$ and the stepsize $\alpha_k \in(0,2)$. The convergence analysis reveals that RABK is extremely effective when it is given a good sampling of the rows into well-conditioned blocks.  Specifically, if
$
\omega_j^k=\frac{\left\|A_{j,:}\right\|_2^2}{\left\|A_{\mathcal{I}_{i_k},:}\right\|_{F}^2}$ with $ j \in \mathcal{I}_{i_k},
$
then \eqref{rabk-g} becomes
\begin{equation}\label{rabk}
	x^{k+1}=x^k-\alpha_k\frac{(A_{\mathcal{I}_{i_k},:})^\top(A_{\mathcal{I}_{i_k},:}x^k-
		b_{\mathcal{I}_{i_k}})}{\|A_{\mathcal{I}_{i_k},:}\|^2_F}.
\end{equation}
Recently, the iteration scheme \eqref{rabk} has  been used in many works.
In the literature \cite{miao2022greedy},  Miao and Wu proposed a greedy randomized average block Kaczmarz method  for solving linear systems.   Necoara \cite{necoara2022stochastic}  used the idea of blocks of sets to develop  accelerated RP methods for convex feasibility problems.
For another block
version of the RK method, we refer to the literatures \cite{Du20Ran,moorman2021randomized} and  the references therein.

\subsection{Our contribution}

This paper extends  the ideas of the greedy probability criterion and the average block approach to solve linear feasibility problems, obtaining a \emph{greedy randomized average block projection} (GRABP) method. Recall that  $\{\mathcal{I}_1,\mathcal{I}_2,\cdots,\mathcal{I}_t\}$ is a partition of the row index set $[m]$ of the matrix $A$.
At each step, we  greedily choose a nonempty index set $\mathcal{U}_k$ using an adaptive thresholding rule so that for any $i \in \mathcal{U}_k$, the norm of the residual $(\langle A_{\mathcal{I}_{i},:},x^{k}\rangle-b_{\mathcal{I}_{i}})_{+}$ should be larger than a prescribed threshold.
After selecting an index $i_k\in\mathcal{U}_k$ with a certain probability criterion, we project the current iteration vector onto each feasible
region $\{x\mid A_{i_k,:}x\leq b_{i_k}\}$ with $i_k\in\mathcal{U}_k$, average them, and apply extrapolated step sizes to construct the GRABP method.
Relying on a lemma due to Hoffman \cite{hoffman2003approximate,Lev10}, two kinds
of extrapolated stepsizes for the GRABP method are analyzed.
The numerical results show the advantage of the GRABP method over several state-of-the-art solvers, such as the RP method, the SKM method, the GSKM method, and the PASKM method.

\subsection{Organization}
The organization of this paper is as follows. We give the GRABP
method for solving the linear  feasibility problems in Section 2 and its convergence analysis in Section 3.
Numerical experimental results are presented in Section 4. Finally, we end this paper with concluding remarks in Section 5.
\section{The greedy randomized average block Kaczmarz method}
In this section, we  introduce the  GRABP method for solving the linear feasibility problem \eqref{inequalities}.
The method is formally described in Algorithm \ref{GRABP}.
\begin{algorithm}
	\caption{(The GRABP method) \label{GRABP}}
	\begin{algorithmic}
		\Require
		$A\in \mathbb{R}^{m\times n}$, $b\in \mathbb{R}^m$, $k=0$, $K$, $t$  and an initial $x^0\in \mathbb{R}^{n}$.
		\begin{enumerate}
			\item[1:]  Let $\{\mathcal{I}_1,\mathcal{I}_2,\cdots,\mathcal{I}_t\}$ be a partition of $[m]$.
			
			{\bf while $k< K$ do}
			\item[2:]  Compute
			\begin{equation}\label{epsilon_k}
				\epsilon_k=\frac{1}{2}\left(\frac{1}{\|(Ax^k-b)_+\|^2_2}\max\limits_{1\leq i\leq t} \frac{\|(A_{\mathcal{I}_i,:}x^k-b_{\mathcal{I}_i})_+\|^2_2}{\|A_{\mathcal{I}_i,:}\|^2_F}+\frac{1}{\| A\|^2_F}\right).
			\end{equation}
			
			\item[3:]   Determine the index set of positive integers
			\begin{equation}\label{Uk}
				\mathcal{U}_k=\left\{i \mid \|(A_{\mathcal{I}_i,:}x^k-b_{I_i})_+\|^2_2 \geq \epsilon_k \|(Ax^k-b)_+\|^2_2 \|A_{\mathcal{I}_i,:}\|^2_F \right\}.
			\end{equation}
			
			\item[4:]  Select $i_k \in \mathcal{U}_k$ according to probability $$\Pr(\text{index}=i)=p_{k,i}, \:\:\:  i=1, 2, \cdots, t,$$
			with $p_{k,i}=0$ if $i\notin \mathcal{U}_k$, $p_{k,i}\geq 0$ if $i\in \mathcal{U}_k$, and $\sum\limits_{i\in \mathcal{U}_k}p_{k,i}=1$.

			\item[5:]  Update
			$$x^{k+1}=x^k-\alpha_k\frac{(A_{\mathcal{I}_{i_k},:})^\top(A_{\mathcal{I}_{i_k},:}x^k-
				b_{\mathcal{I}_{i_k}})_+}{\|A_{\mathcal{I}_{i_k},:}\|^2_F},$$
			and set  $k=k+1$.
			
		\end{enumerate}
		\Ensure
		The approximate solution $x^{K}$.
	\end{algorithmic}
\end{algorithm}

\begin{remark}
	As done in the literatures \cite{bai2018relaxed,zhang2020relaxed}, we can introduce an arbitrary relaxation parameter $\theta\in[0,1]$ into $\epsilon_k$ in Algorithm \ref{GRABP}, i.e.,
	$$
	\epsilon_k=\frac{\theta}{\|(Ax^k-b)_+\|^2_2}\max\limits_{1\leq i\leq t} \frac{\|(A_{\mathcal{I}_i,:}x^k-b_{\mathcal{I}_i})_+\|^2_2}{\|A_{\mathcal{I}_i,:}\|^2_F}+(1-\theta)\frac{1}{\| A\|^2_F}.
	$$
	In this case, setting $\theta=1$, we can obtain a greedy average block projection method. In fact,
	the parameter $\theta$ affects the index set $\mathcal{U}_k$ to some extent, but the algorithm maintains linear convergence regardless of the value chosen for the parameter $\theta$. This paper focuses on  the case where $\theta=\frac{1}{2}$.
\end{remark}

We next present some specific details of Algorithm \ref{GRABP}.

\noindent\textbf{ Randomized row partition.}
In the setup of Algorithm \ref{GRABP}, we need to partition the row index set of $[m]$ of the coefficient matrix $A$ into $\{\mathcal{I}_1,\mathcal{I}_2,\cdots,\mathcal{I}_t\}$.
The row partition of the matrix has been extensively discussed in the literatures \cite{tropp2009column,Nec19,necoara2022stochastic,xie2021subset}. In this paper, we use a simple partitioning strategy, i.e.,
\begin{equation}\label{Ik}
	\mathcal{I}_i=\left\{\varpi(k): k=\lfloor (i-1)m/t\rfloor+1, \lfloor (i-1)m/t\rfloor+2,\ldots,\lfloor im/t\rfloor\right\}, i=1, 2, \ldots, t,
\end{equation}
where $\varpi$ is a permutation on $[m]$ chosen uniformly at random, and $\lfloor \cdot \rfloor$ denotes the largest integer which is smaller than or equal to a certain number.

\noindent\textbf{ Probability strategy.} There are many choices for the probability strategy in step 4 of Algorithm \ref{GRABP}. Let $p>0$, $\mu>0$ and  for $k=0, 1, 2,\cdots$, $$\tilde{r}^k_{I_i}=\left\{\begin{array}{ll}
(A_{I_i,:}x^k-b_{I_i})_+,& \text{if} \ i\in \mathcal{U}_k, \\
0,& \ \text{otherwise}.
\end{array}\right.$$
Here we use the following two different probability strategies:
\begin{equation}\label{pki1}
	\begin{aligned}
		p_{k,i}=\left\{\begin{array}{ll}\frac{\|\tilde{r}^k_{\mathcal{I}_i}\|_p^p}{ \sum\limits_{i\in \mathcal{U}_k}\|\tilde{r}^k_{\mathcal{I}_i} \|^p_{p}}, &\text{if} \ i\in \mathcal{U}_k, \\
			0,& \text{otherwise},
		\end{array}\right.
	\end{aligned}
\end{equation}
and
\begin{equation}\label{pki2}
	\begin{aligned}
		p_{k,i}=\left\{\begin{array}{ll}\frac{\|\tilde{r}^k_{\mathcal{I}_i}\|_2^\mu}{ \sum\limits_{i\in \mathcal{U}_k}\|\tilde{r}^k_{\mathcal{I}_i} \|^\mu_2},& \text{if} \ i\in \mathcal{U}_k, \\
			0,& \text{otherwise}.
		\end{array}\right.
	\end{aligned}
\end{equation}
Obviously, both (\ref{pki1}) and (\ref{pki2}) satisfy   $p_{k,i}=0$ if $i\notin \mathcal{U}_k$, $p_{k,i}\geq 0$ if $i\in \mathcal{U}_k$, and $\sum\limits_{i\in \mathcal{U}_k} p_{k,i}=1$.
When $p = u = 2$,  the above two probability strategies are the same.
In subsequent proofs and analyses, we find that the GRABP method converges regardless of the probability strategy chosen.

\noindent\textbf{ The choices of stepsize.}
It is well-known that the stepsize affects the convergence of the algorithm.
Here, we focus on two choices of the iteration stepsize  $\alpha_k$. One is a constant stepsize, i.e.,  $\alpha_k$ is equal to  $\alpha \in (0,\frac{2}{\zeta})$ with
\begin{equation}\label{xi}
	\zeta=\max\limits_{1\leq i\leq t}\frac{\sigma^2_{\max}(A_{\mathcal{I}_i,:})}{\|A_{\mathcal{I}_i,:}\|^2_F}.
\end{equation}
Another is adaptive stepsize
\begin{equation}\label{alpha}
	\alpha_k=w\frac{\|(A_{\mathcal{I}_{i_k},:}x^k-b_{\mathcal{I}_{i_k}})_+\|^2_2\| A_{\mathcal{I}_{i_k},:}\|^2_F}{\|(A_{\mathcal{I}_{i_k},:})^\top(A_{\mathcal{I}_{i_k},:}x^k-b_{\mathcal{I}_{i_k}})_+\|^2_2}  \end{equation}
with $w\in(0,2)$.
Since at the $k$-th iterate of the GRABP method, it  always hold that
$\left\|(A_{\mathcal{I}_{i_k,:}}x^k-b_{\mathcal{I}_{i_k}})_+\right\|^2_2\neq0,$
and then we have
\begin{equation}
	\label{equ-1025}
	\left\|(A_{\mathcal{I}_{i_k},:})^\top(A_{\mathcal{I}_{i_k},:}x^k-b_{\mathcal{I}_{i_k}})_+\right\|^2_2\neq0.
\end{equation}
Indeed, if \eqref{equ-1025} does not hold, i.e.
$
(A_{\mathcal{I}_{i_k},:})^\top(A_{\mathcal{I}_{i_k},:}x^k-b_{\mathcal{I}_{i_k}})_+=0
$.
Then it holds that $$\langle A_{\mathcal{I}_{i_k},:}(x^k-P_S( x^{k})), (A_{\mathcal{I}_{i_k},:}x^k-b_{\mathcal{I}_{i_k}})_+\rangle=0.$$
Noting that $P_S( x^{k})\in S$, hence we have $A_{\mathcal{I}_{i_k},:}P_S( x^{k})\leq b_{\mathcal{I}_{i_k}}$. So
$$
\begin{aligned}
0&=\langle A_{\mathcal{I}_{i_k},:}(x^k-P_S( x^{k})), (A_{\mathcal{I}_{i_k},:}x^k-b_{\mathcal{I}_{i_k}})_+\rangle
\\
&\geq\langle A_{\mathcal{I}_{i_k},:} x^k-b_{\mathcal{I}_{i_k}}, (A_{\mathcal{I}_{i_k},:}x^k-b_{\mathcal{I}_{i_k}})_+\rangle
\\
&=\left\|(A_{\mathcal{I}_{i_k,:}}x^k-b_{\mathcal{I}_{i_k}})_+\right\|^2_2,
\end{aligned}$$
which is contrary to the fact that $\left\|(A_{\mathcal{I}_{i_k,:}}x^k-b_{\mathcal{I}_{i_k}})_+\right\|^2_2\neq 0$.
So the adaptive step size in (\ref{alpha}) is well defined.
For convenience, we refer to the GRABP method with  constant stepsize as GRABP-c. The GRABP method with  adaptive stepsize, we refer to as GRABP-a.
In Section \ref{sect:3}, we will discuss the convergence properties of the GRABP-c and GRABP-a methods, respectively.

Finally, let us briefly state that the index set $\mathcal{U}_k$ in (\ref{Uk}) is well defined, i.e., the index set $\mathcal{U}_k$ is nonempty. Indeed,
assume that the index $\mathcal{I}_{i_k}$ satisfies
$$\frac{\left\|\left(A_{\mathcal{I}_{i_k},:} x^k-b_{\mathcal{I}_{i_k}}\right)_{+}\right\|_2^2}{\left\|A_{\mathcal{I}_{i_k},:}\right\|_F^2}:=\max _{1 \leq i \leq t} \frac{\left\|\left(A_{\mathcal{I}_i,:} x^k-b_{\mathcal{I}_i}\right)_{+}\right\|_2^2}{\left\|A_{\mathcal{I}_i,:}\right\|_F^2}.$$
One can verified that
$$
\begin{aligned}
\frac{\left\|\left(A_{\mathcal{I}_{i_k},:} x^k-b_{\mathcal{I}_{i_k}}\right)_{+}\right\|_2^2}{\left\|A_{\mathcal{I}_{i_k},:}\right\|_F^2} & \geq \sum_{i=1}^t \frac{\left\|A_{\mathcal{I}_i,:}\right\|_F^2}{\|A\|_F^2} \frac{\left\|\left(A_{\mathcal{I}_i,:} x^k-b_{\mathcal{I}_i}\right)_{+}\right\|_2^2}{\left\|A_{\mathcal{I}_i,:}\right\|_F^2} \\
&=\frac{\left\|\left(A x^k-b\right)_{+}\right\|_2^2}{\|A\|_F^2},
\end{aligned}
$$
and
$$
\begin{aligned}
\frac{\left\|\left(A_{\mathcal{I}_{i_k},:} x^k-b_{\mathcal{I}_{i_k}}\right)_{+}\right\|_2^2}{\left\|A_{\mathcal{I}_{i_k},:}\right\|_F^2} & \geq \frac{1}{2} \max _{1 \leq i \leq t} \frac{\left\|\left(A_{\mathcal{I}_i,:} x^k-b_{\mathcal{I}_i}\right)_{+}\right\|_2^2}{\left\|A_{\mathcal{I}_i,:}\right\|_F^2}+\frac{1}{2} \frac{\left\|\left(A x^k-b\right)_{+}\right\|_2^2}{\|A\|_F^2} \\
&=\epsilon_k\left\|\left(A x^k-b\right)_{+}\right\|_2^2.
\end{aligned}
$$
This implies that the index $i_k$ always belongs to $\mathcal{U}_k$ and hence the index set $\mathcal{U}_k$ will always  be nonempty.
In addition, we note that the index set $\mathcal{U}_k$ is flexible during the iteration, i.e., it changes as the number of iteration steps $k$ increases.

\section{Convergence analysis}
\label{sect:3}

In this section, we will discuss the convergence property of Algorithm \ref{GRABP}. Let us first introduce a crucial lemma.

\begin{lemma}[Hoffman\cite{hoffman2003approximate}]\label{Hoffman}
	Let $x\in \mathbb{R}^{n} $ and $S$ be the feasible region of the linear feasibility problem (\ref{inequalities}). Then, there exists a constant $L>0$ such that the following identity holds:
	$$ \| x-P_S( x)\|^2_2 \leq L^2 \|(Ax-b)_+\|^2_2,$$
	where $P_S(x)$ represents the orthogonal projection of $x$ onto the  feasible  region $S$.
\end{lemma}
Lemma \ref{Hoffman} is a well-known result of Hoffman on systems of linear inequalities. The constant $L$ is called the Hoffman constant. We will use Lemma \ref{Hoffman} to establish two convergence theorems about the  GRABP algorithm.
For the GRABP method with constant stepsize, we have the following convergence result.

\begin{theorem}\label{convergence_1}
	Suppose that the linear feasibility problem (\ref{inequalities}) is consistent, i.e., the feasible region $S$ is nonempty, and the stepsize $\alpha_k$ of the $k$-th iteration of the GRABP method is a constant $\alpha \in (0,\frac{2}{\zeta})$ with $\zeta$ defined as in (\ref{xi}). Then the iteration sequence  $\{ x^k\}^\infty_{k=0}$ generated by the GRABP-c method satisfies
	$$
	\mathbb{E}\left[\left\| x^{k}-P_S( x^{k})\right\|^2_2 \right] \leq
	\left(1-\frac{2\alpha-\alpha^2\zeta}{L^2\| A \|^2_F}\right)^k\left\| x^{0}-P_S( x^{0})\right\|^2_2,
	$$
	where $P_S(x^k)$ represents the orthogonal projection of $x^k$ onto the  feasible  region $S$.
\end{theorem}
\begin{proof}
	Straightforward calculations yield
	\begin{equation}\label{5}
		\begin{aligned}
			\| x^{k+1}-P_S( x^{k+1})\|^2_2 \leq &\| x^{k+1}-P_S( x^{k})\|^2_2\\
			=& \bigg\| x^k-\alpha\frac{(A_{\mathcal{I}_{i_k},:})^\top(A_{\mathcal{I}_{i_k},:}x^k-b_{\mathcal{I}_{i_k}})_+}{\|A_{\mathcal{I}_{i_k},:}\|^2_F}-P_S( x^{k})\bigg\|^2_2\\
			= &\| x^{k}-P_S( x^{k})\|^2_2+\alpha^2\frac{\|(A_{\mathcal{I}_{i_k},:})^\top(A_{\mathcal{I}_{i_k},:}x^k-b_{\mathcal{I}_{i_k}})_+\|^2_2}{\|A_{\mathcal{I}_{i_k},:}\|^4_F}
			\\
			&-\frac{2\alpha}{\|A_{\mathcal{I}_{i_k},:}\|^2_F}\big\langle(A_{\mathcal{I}_{i_k},:})^\top(A_{\mathcal{I}_{i_k},:}x^k-b_{\mathcal{I}_{i_k}})_+,x^{k}-P_S( x^{k})\big\rangle.
		\end{aligned}
	\end{equation}
	Noting that $P_S( x^{k})\in S$, hence we have $A_{\mathcal{I}_{i_k},:}P_S( x^{k})\leq b_{\mathcal{I}_{i_k}}$. So
	\begin{equation}\label{6}
		\begin{aligned}
			\langle(A_{\mathcal{I}_{i_k},:})^\top(A_{\mathcal{I}_{i_k},:}x^k-b_{\mathcal{I}_{i_k}})_+, x^{k}-P_S( x^{k})\rangle&=\langle(A_{\mathcal{I}_{i_k},:}x^k-b_{\mathcal{I}_{i_k}})_+, A_{\mathcal{I}_{i_k},:}(x^{k}-P_S( x^{k}))\rangle\\
			&\geq\langle(A_{I_{i_k},:}x^k-b_{I_{i_k}})_+,A_{I_{i_k},:}x^{k}-b_{I_{i_k}}\rangle\\
			&=\|(A_{\mathcal{I}_{i_k},:}x^k-b_{\mathcal{I}_{i_k}})_+\|^2_2.
		\end{aligned}
	\end{equation}
	It follows from \eqref{xi} that
	\begin{equation}\label{7}
		\begin{aligned}
			\frac{\|(A_{\mathcal{I}_{i_k},:})^\top(A_{\mathcal{I}_{i_k},:}x^k-b_{\mathcal{I}_{i_k}})_+\|^2_2}{\|A_{\mathcal{I}_{i_k},:}\|^4_F}
			&\leq\frac{\sigma^2_{\max}(A_{\mathcal{I}_{i_k},:})}{\|A_{\mathcal{I}_{i_k},:}\|^2_F}\frac{\|(A_{\mathcal{I}_{i_k},:}x^k-b_{\mathcal{I}_{i_k}})_+\|^2_2}{\|A_{\mathcal{I}_{i_k},:}\|^2_F}\\
			&\leq\zeta\frac{\|(A_{\mathcal{I}_{i_k},:}x^k-b_{\mathcal{I}_{i_k}})_+\|^2_2}{\|A_{\mathcal{I}_{i_k},:}\|^2_F}.
		\end{aligned}
	\end{equation}
	Substituting  (\ref{6}) and (\ref{7}) into  (\ref{5}), we obtain
	$$\| x^{k+1}-P_S( x^{k+1})\|^2_2\leq \| x^{k}-P_S( x^{k})\|^2_2-(2\alpha-\alpha^2\zeta)\frac{\|(A_{\mathcal{I}_{i_k},:}x^k-b_{\mathcal{I}_{i_k}})_+\|^2_2}{\|A_{\mathcal{I}_{i_k},:}\|^2_F}. $$
	By taking conditional expectation on both sides of this inequality, we get
	$$\begin{aligned}
	\mathbb{E}_k\left[\| x^{k+1}-P_S( x^{k+1})\|^2_2 \right]
	&\leq \| x^{k}-P_S( x^{k})\|^2_2-(2\alpha-\alpha^2\zeta)\mathbb{E}_k\left[\frac{\|(A_{\mathcal{I}_{i_k},:}x^k-b_{\mathcal{I}_{i_k}})_+\|^2_2}{\|A_{\mathcal{I}_{i_k},:}\|^2_F}\right]\\
	&=\| x^{k}-P_S( x^{k})\|^2_2-(2\alpha-\alpha^2\zeta)\sum\limits_{i_k\in \mathcal{U}_k}p_{k,i}\frac{\|(A_{\mathcal{I}_{i_k},:}x^k-b_{\mathcal{I}_{i_k}})_+\|^2_2}{\|A_{\mathcal{I}_{i_k},:}\|^2_F}\\
	&\overset{(a)}\leq \| x^{k}-P_S( x^{k})\|^2_2-(2\alpha-\alpha^2\zeta)\epsilon_k \|(Ax^k-b)_+\|^2_2\\
	&=\left(1-\left(2\alpha-\alpha^2\zeta\right)\frac{\epsilon_k \|(Ax^k-b)_+\|^2_2}{\| x^{k}-P_S( x^{k})\|^2_2}\right)\| x^{k}-P_S( x^{k})\|^2_2 \\
	&\overset{(b)}\leq \left(1-\frac{(2\alpha-\alpha^2\zeta)\epsilon_k }{L^2}\right)\| x^{k}-P_S( x^{k})\|^2_2,
	\end{aligned}
	$$
	where $(a)$ follows from the definition of $\mathcal{U}_k$ and $2\alpha-\alpha^2\zeta>0$ and $(b)$ follows from the Hoffman bound.
	
	In view of the definition of $\epsilon_k$ in (\ref{epsilon_k}), we obtain
	$$
	\begin{aligned}
	\epsilon_k\| A \|^2_F
	&=\frac{1}{2}\frac{\| A \|^2_F }{\|(Ax^k-b)_+\|^2_2}\max\limits_{1\leq i\leq t}\frac{\|(A_{\mathcal{I}_i,:}x^k-b_{\mathcal{I}_i})_+\|^2_2}{\|A_{\mathcal{I}_i,:}\|^2_F}+\frac{1}{2}\\
	&=\frac{\max\limits_{1\leq i\leq t}\frac{\|(A_{\mathcal{I}_i,:}x^k-b_{\mathcal{I}_i})_+\|^2_2}{\|A_{\mathcal{I}_i,:}\|^2_F}}{2\sum\limits_{i=1}^t \frac{\|A_{\mathcal{I}_i,:}\|^2_F}{\| A\|^2_F }\frac{\|(A_{\mathcal{I}_i,:}x^k-b_{\mathcal{I}_i})_+\|^2_2}{\|A_{\mathcal{I}_i,:}\|^2_F}}+\frac{1}{2}\\
	&\geq\frac{1}{2}+\frac{1}{2}\\
	&=1. \notag
	\end{aligned}
	$$
	Thus, we have
	\begin{equation}\label{10}
		\mathbb{E}_k\left[\| x^{k+1}-P_S( x^{k+1})\|^2_2 \right]
		\leq \left(1-\frac{2\alpha-\alpha^2\zeta}{L^2\| A \|^2_F}\right)\| x^{k}-P_S( x^{k})\|^2_2.
	\end{equation}
	By taking full expectation on both sides of the inequality $(\ref{10})$, we have
	$$
	\mathbb{E}\left[\| x^{k+1}-P_S( x^{k+1})\|^2_2 \right]
	\leq \left(1-\frac{2\alpha-\alpha^2\zeta}{L^2\| A \|^2_F}\right)\mathbb{E}\left[\| x^{k}-P_S( x^{k})\|^2_2\right].
	$$
	By induction on the iteration index $k$, we can obtain the desired result.
\end{proof}

Next, we analyze the convergence of the GRABP method with adaptive stepsize.

\begin{theorem}\label{convergence_2}
	Suppose that the linear feasibility problem (\ref{inequalities}) is consistent, i.e., the feasible region $S$ is nonempty, and the stepsize $\alpha_k$ of the $k$-th iteration of the GRABP method is  chosen as in  (\ref{alpha})  with $w\in(0,2)$. Then the iteration sequence  $\{ x^k\}^\infty_{k=0}$ generated by the GRABP-a method satisfies
	$$
	\mathbb{E}[\| x^{k}-P_S( x^{k})\|^2_2 ]
	\leq \left(1-\frac{2w-w^2}{\zeta L^2\| A \|^2_F}\right)^k\| x^{0}-P_S( x^{0})\|^2_2,
	$$
	where $P_S(x^{k})$ represents the orthogonal projection of $x^k$ onto the  feasible  region $S$ and $\zeta$ is defined as  \eqref{xi}.
	
\end{theorem}
\begin{proof}
	By using similar arguments as that in \eqref{5}, we have
	$$
	\begin{aligned}
	\| x^{k+1}-P_S( x^{k+1})\|^2_2 \leq& \| x^{k}-P_S( x^{k})\|^2_2+\alpha_k^2\frac{\|(A_{\mathcal{I}_{i_k},:})^\top(A_{\mathcal{I}_{i_k},:}x^k-b_{\mathcal{I}_{i_k}})_+\|^2_2}{\|A_{\mathcal{I}_{i_k},:}\|^4_F}\\
	&-\frac{2\alpha_k}{\|A_{\mathcal{I}_{i_k},:}\|^2_F}\langle(A_{\mathcal{I}_{i_k},:})^\top(A_{\mathcal{I}_{i_k},:}x^k-b_{\mathcal{I}_{i_k}})_+,x^{k}-P_S( x^{k})\rangle\\
	\leq& \|x^{k}-P_S(x^{k})\|^2_2+\alpha_k^2\frac{\|(A_{\mathcal{I}_{i_k},:})^\top(A_{\mathcal{I}_{i_k},:}x^k-b_{\mathcal{I}_{i_k}})_+\|^2_2}{\|A_{\mathcal{I}_{i_k},:}\|^4_F}\\
	&-\frac{2\alpha_k\|(A_{\mathcal{I}_{i_k},:}x^k-b_{\mathcal{I}_{i_k}})_+\|^2_2}{\|A_{\mathcal{I}_{i_k},:}\|^2_F}.
	\end{aligned}
	$$
	Substituting  $\alpha_k$ into this equality, we obtain
	$$
	\begin{aligned}
	&\| x^{k+1}-P_S( x^{k+1})\|^2_2
	\\
	&\leq \| x^{k}-P_S(x^{k})\|^2_2
	-(2w-w^2)\frac{\|(A_{\mathcal{I}_{i_k},:}x^k-b_{\mathcal{I}_{i_k}})_+\|^2_2\| A_{\mathcal{I}_{i_k},:}\|^2_F}{\|(A_{\mathcal{I}_{i_k},:})^\top(A_{\mathcal{I}_{i_k},:}x^k-b_{\mathcal{I}_{i_k}})_+\|^2_2}
	\frac{\|(A_{\mathcal{I}_{i_k},:}x^k-b_{\mathcal{I}_{i_k}})_+\|^2_2}{\|A_{\mathcal{I}_{i_k},:}\|^2_F}.
	\end{aligned}$$
	Since $2w-w^2>0$  and
	$$
	\|(A_{\mathcal{I}_{i_k},:})^\top(A_{\mathcal{I}_{i_k},:}x^k-b_{\mathcal{I}_{i_k}})_+\|^2_2
	\leq\sigma^2_{\max}(A_{\mathcal{I}_{i_k},:})\|(A_{\mathcal{I}_{i_k},:}x^k-b_{\mathcal{I}_{i_k}})_+\|^2_2,
	$$
	we have
	$$\| x^{k+1}-P_S( x^{k+1})\|^2_2\leq \| x^{k}-P_S( x^{k})\|^2_2-\frac{2w-w^2}{\zeta}\frac{\|(A_{\mathcal{I}_{i_k},:}x^k-b_{\mathcal{I}_{i_k}})_+\|^2_2}{\|A_{\mathcal{I}_{i_k},:}\|^2_F}.
	$$
	Therefore, we can obtain
	\begin{equation}\label{12}
		\begin{split}
			\mathbb{E}_k\big[\| x^{k+1}-P_S( x^{k+1})\|^2_2 \big]
			&\leq \| x^{k}-P_S( x^{k})\|^2_2-\frac{2w-w^2}{\zeta}\mathbb{E}_k\left[\frac{\|(A_{\mathcal{I}_{i_k},:}x^k-b_{\mathcal{I}_{i_k}})_+\|^2_2}{\|A_{\mathcal{I}_{i_k},:}\|^2_F}\right]\\
			&=\| x^{k}-P_S( x^{k})\|^2_2-\frac{2w-w^2}{\zeta}\sum\limits_{i_k\in \mathcal{U}_k}p_{k,i}\frac{\|(A_{\mathcal{I}_{i_k},:}x^k-b_{\mathcal{I}_{i_k}})_+\|^2_2}{\|A_{\mathcal{I}_{i_k},:}\|^2_F}\\
			&\overset{(c)}\leq \| x^{k}-P_S( x^{k})\|^2_2-\frac{2w-w^2}{\zeta}\epsilon_k \|(Ax^k-b)_+\|^2_2\\
			&=\left(1-\frac{2w-w^2}{\zeta}\frac{\epsilon_k \|(Ax^k-b)_+\|^2_2}{\| x^{k}-P_S( x^{k})\|^2_2}\right)\| x^{k}-P_S( x^{k})\|^2_2 \\
			&\leq \left(1-\frac{2w-w^2}{\zeta L^2\| A \|^2_F}\right)\| x^{k}-P_S( x^{k})\|^2_2,
		\end{split}
	\end{equation}
	where $(c)$ follows from the definition of $\mathcal{U}_k$ and the fact that $2w-w^2>0$.
	By taking full expectation on both sides of $(\ref{12})$,  we get
	$$
	\mathbb{E}[\| x^{k+1}-P_S( x^{k+1})\|^2_2 ]
	\leq \left(1-\frac{2w-w^2}{\zeta L^2\| A \|^2_F}\right)\mathbb{E}[\| x^{k}-P_S( x^{k})\|^2_2].
	$$
	We can obtain the desired result by induction on the iteration index $k$.
\end{proof}

\begin{remark}  It can be seen from  Theorem \ref{convergence_1} that the convergence factor of the GRABP-c method is $1-\frac{2\alpha-\alpha^2\zeta}{L^2\| A \|^2_F}$, and it reaches the minimum value $1-\frac{1}{\zeta L^2\| A \|^2_F}$ when $\alpha=\frac{1}{\zeta }$.
	Similarly, from Theorem \ref{convergence_2}, the convergence factor of the GRABP-a method is $1-\frac{2w-w^2}{\zeta L^2\| A \|^2_F}$, which reaches the minimum value $1-\frac{1}{\zeta L^2\| A \|^2_F}$ when $w=1$.
\end{remark}

\section{Experimental results}

In this section, we perform numerical experiments to show the computational efficiency of the GRABP algorithm (Algorithms GRABP-c and GRABP-a), and compare   the number of iteration steps (denoted by ``IT") and the computing time in seconds (denoted by ``CPU") with those of the RP method, the SKM method, the GSKM method, and the PASKM method. For a fair comparison, we run these algorithms 10 times and give the average performance of the experiments. All experiments are carried out using MATLAB on a personal computer ( Intel(R) Core(TM) i7-8700 CPU @3.20GHz 3.19 GHz ).

\subsection{Numerical setup}

To analyze computational performance, we perform  numerical experiments for a wide range of instances, including  randomly generated problems and real-world test problems, where the real-world test instances are the SuiteSparse Matrix Collection \cite{kolodziej2019suitesparse}  and the sparse Netlib LP instances  \cite{netlie}. We use $x^0=0\in\mathbb{R}^n$ as an initial point. In testing randomly generated problems and the SuiteSparse Matrix Collection, the experiments are terminated when the relative solution error (RES) at $x^k$ is less than $10^{-8}$, where RES is defined by $$\text{RES}=\frac{\| (Ax^k-b)_+\|_2}{\| b\|_2}.$$  When testingthe sparse Netlib LP instances, we set the stopping criterion for the experiment to be $$\frac{max(Ax^k-b)}{ max(Ax^0-b)}\leq \phi,$$ where $\phi$ is the tolerance gap. In addition, when the coputation time is greater than 50 seconds, i.e., CPU$>50$, the algorithm  forces the forced.

During our test, we use the row indices $\mathcal{I}_i$ of the random partition $\{\mathcal{I}_1, \mathcal{I}_2, \cdots, \mathcal{I}_t\}$ defined as in (\ref{Ik}).
We test the randomly generated problem by dividing the rows of the matrix into ten blocks, i.e., $t = 10$.
In testing the SuiteSparse Matrix Collection, we take $t=\lceil\| A\|^2_2\rceil$.
In testing the sparse Netlib LP instances, we take $t=5$.
The probability criterion $p_{k,i}$ in Step 4 of the GRABP method is chosen as \eqref{pki1} with $p=2$.
For the constant stepsizes of the GRABP-c, we use $\alpha=\frac{1}{\zeta}$ and $\alpha=\frac{1.95}{\zeta}$ with  $\zeta$ being defined as in (\ref{xi}). For the adaptive stepsizes of the GRABP-a, we use $\alpha_k$ being defined as in (\ref{alpha}) with $w=1$ and $w=1.95$.
The SKM, GSKM, and PASKM algorithms involve the selection of many parameters as well, and we have selected a set of parameters with better performance based on the literature \cite{morshed2021sampling}. The parameters of the PASKM algorithm are selected as the PASKM-2 algorithm (see the literature \cite{morshed2021sampling} for details).

\subsection{Experiments on randomly generated instances }
For the randomly generated coefficient matrix $A$, we mainly consider two types, namely dense and sparse matrices. We randomly generate the dense matrix by the MATLAB function ``{\tt randn}". The sparse matrix is generated randomly by the MATLAB function ``{\tt sprandn}" with a density of $\frac{1}{2log(mn)}$ for the non-zero elements.
To ensure that the system (\ref{inequalities}) is consistent, i.e., $S\neq\emptyset$, we randomly generate vectors $x_1\in \mathbb{R}^{n} $, $x_2\in \mathbb{R}^{n}$, $x_3\in \mathbb{R}^{m} $ and  set the right-hand side as $b=0.5Ax_1+0.5Ax_2+x_3$. Both $x_1$ and $x_2$  are generated randomly by the MATLAB function ``{\tt randn}". The vector $x_3$  is a randomly generated vector with elements in the range $[0.1,1]$.

From Table \ref{Tabdense_n} to Table \ref{Tabsparse}, we tested the performance of all algorithms when the coefficient matrix is dense and sparse, respectively.
We test two sets of coefficient matrices with a constant number of rows but an increasing number of columns in Tables \ref{Tabdense_n} and \ref{Tabsprse_n}. Regardless of whether the coefficient matrix is dense or sparse, the number of iteration steps and the computational time increases with the number of matrix columns for all methods except for the  GRABP-c algorithm.
Tables \ref{Tabdense} and \ref{Tabsparse} show the performance of the algorithms at different orders of the coefficient matrix.
From the Tables, we can observe that all algorithms performed well.
The GRABP-c algorithm with $\alpha=\frac{1.95}{\zeta}$ performs better than the  GRABP-c algorithm with $\alpha=\frac{1}{\zeta}$, while  the GRABP-a algorithm with $w=1.95$ performs better than the  GRABP-a algorithm with $w=1$. In most cases, the GRABP-a algorithm with $w=1.95$ takes the least time for computation.

\begin{table*}[htp]
\footnotesize
\centering
\caption{ \title{\scriptsize Numerical results for $m$-by-$n$ random dense  matrices $A$ with $m=5000$ and different $n$.}}
\label{Tabdense_n}
\renewcommand\arraystretch{1.1}
\setlength{\tabcolsep} {1.8mm} {
\begin{tabular}{l|lllllll}
\toprule
\multicolumn{2}{c}{$n$}  &$100$  &$200$  &$300$   &$400$   &$500$   &$600$\\
\midrule
\multirow{2}{*}{RP}
&IT        &75987.4       &103241.8     &134719.6     &164177.6      &209444.5      &236523        \\
&CPU       &2.5778e+00    &3.5595e+00   &4.6618e+00   &5.8100e+00    &7.5766e+00    &8.7382e+00  \\
\midrule
\multirow{2}{*}{SKM}
&IT        &1358.8        &2533.1       &4121.5       &5862          &8987.8        &11486.6       \\
&CPU       &4.4360e-01    &9.6030e-01   &1.8132e+00   &3.0207e+00    &5.1937e+00    &7.3037e+00    \\
\midrule
\multirow{2}{*}{GSKM}
&IT        &1239.6        &2018.4       &3159.5       &4832.6        &7343.3        &9573.8        \\
&CPU       &3.8290e-01    &8.0590e-01   &1.7159e+00   &3.6972e+00    &7.2677e+00    &1.1722e+01   \\
\midrule
\multirow{2}{*}{PASKM}
&IT        &433.8         &752          &971.8        &1308          &1603.6        &1874.8       \\
&CPU  &\bf{1.3050e-01}    &3.2420e-01   &5.3290e-01   &1.0126e+00    &1.5878e+00    &2.3120e+00  \\
\midrule
\multirow{1}{*}{GRABP-c}
&IT        &2294.9        &2078.4       &1946.7       &1977.7        &2154.9        &2087.9    \\
\multirow{1}{*}{$(\alpha=\frac{1}{\zeta})$}
&CPU       &7.5010e-01    &9.6030e-01   &1.2090e+00   &1.4622e+00    &1.9276e+00    &2.2581e+00   \\
\midrule
\multirow{1}{*}{GRABP-c}
&IT        &1111.7        &948.7        &924.1        &926.8         &992.7         &937.3   \\
\multirow{1}{*}{$(\alpha=\frac{1.95}{\zeta})$}
&CPU       &4.6450e-01    &5.3740e-01   &7.2270e-01   &8.7230e-01    &1.1098e+00    &1.2579e+00   \\
\midrule
\multirow{1}{*}{GRABP-a}
&IT        &75.7          &105.4        &152          &192.3         &248.8         &258.4     \\
\multirow{1}{*}{$(w=1)$}
&CPU       &1.8360e-01    &2.3550e-01   &3.0020e-01   &3.5460e-01    &4.3930e-01    &5.2570e-01   \\
\midrule
\multirow{1}{*}{GRABP-a}
 &IT    &\bf{28.3}     &\bf{43.7}       &\bf{55.2}       &\bf{65.2}      &\bf{77.4}       &\bf{88.5} \\
\multirow{1}{*}{$(w=1.95)$}
&CPU    &1.7900e-01    &\bf{2.1230e-01} &\bf{2.5060e-01} &\bf{2.8730e-01}&\bf{3.2500e-01} &\bf{3.7540e-01}      \\
\bottomrule
\end{tabular}}
\end{table*}

\begin{table*}[htp]
\footnotesize
\centering
\caption{ \title{\scriptsize Numerical results for $m$-by-$n$ random sparse  matrices $A$ with $m=5000$ and different $n$.}}
\label{Tabsprse_n}
\renewcommand\arraystretch{1.1}
\setlength{\tabcolsep} {1.8mm} {
\begin{tabular}{l|lllllll}
\toprule
\multicolumn{2}{c}{$n$}  &$100$  &$200$  &$300$   &$400$   &$500$   &$600$\\
\midrule
\multirow{2}{*}{RP}
&IT        &102643.8      &163138.9     &223343.9     &233383.9      &317693.8      &294841        \\
&CPU       &3.4062e+00    &5.5635e+00   &7.6438e+00   &8.2230e+00    &1.1406e+01    &1.0865e+01  \\
\midrule
\multirow{2}{*}{SKM}
&IT        &1870.1        &3218.3       &3905.2       &5453.3        &8414.8        &9311.9       \\
&CPU       &5.6240e-01    &1.1410e+00   &1.7219e+00   &2.7823e+00    &4.8358e+00    &6.0915e+00    \\
\midrule
\multirow{2}{*}{GSKM}
&IT        &1476.3        &2673.4       &3147.3       &4548.5        &6729.9        &7591.2       \\
&CPU       &4.5570e-01    &1.0152e+00   &1.6571e+00   &3.4752e+00    &6.7329e+00    &9.5218e+00   \\
\midrule
\multirow{2}{*}{PASKM}
&IT        &537.9         &1346.4       &1829.9       &2614.9        &4205.7        &4637.8       \\
&CPU  &\bf{1.6740e-01}    &5.1300e-01   &9.7920e-01   &1.9981e+00    &4.2786e+00    &5.7903e+00  \\
\midrule
\multirow{1}{*}{GRABP-c}
&IT        &4334.9        &3172.4       &4228.2       &3599.4        &3680.6        &3323.4    \\
\multirow{1}{*}{$(\alpha=\frac{1}{\zeta})$}
&CPU       &1.3123e+00    &1.2356e+00   &2.2755e+00   &2.3735e+00    &3.0562e+00    &3.3482e+00   \\
\midrule
\multirow{1}{*}{GRABP-c}
&IT        &2214.9         &1526.5      &1999.4       &1671.1        &1839          &1580.2   \\
\multirow{1}{*}{$(\alpha=\frac{1.95}{\zeta})$}
&CPU       &7.5320e-01    &7.0260e-01   &1.2315e+00   &1.3005e+00    &1.7398e+00    &1.8413e+00   \\
\midrule
\multirow{1}{*}{GRABP-a}
&IT        &74.9         &98.1          &110.8        &139.2         &166.2         &177.1     \\
\multirow{1}{*}{$(w=1)$}
&CPU       &1.8310e-01    &2.2870e-01   &2.7940e-01   &3.3310e-01    &3.7100e-01    &4.3130e-01   \\
\midrule
\multirow{1}{*}{GRABP-a}
 &IT    &\bf{26.3}    &\bf{37.2}     &\bf{46}         &\bf{51.6}      &\bf{58.9}       &\bf{62.8} \\
\multirow{1}{*}{$(w=1.95)$}
&CPU    &1.7100e-01 &\bf{2.1410e-01} &\bf{2.4490e-01} &\bf{2.7920e-01}&\bf{3.0930e-01} &\bf{3.3700e-01}      \\
\bottomrule
\end{tabular}}
\end{table*}

\begin{table*}[htp]
\footnotesize
\centering
\caption{ \title{\small Numerical results for $m$-by-$n$ random dense  matrices $A$.}}
\label{Tabdense}
\renewcommand\arraystretch{1.1}
\setlength{\tabcolsep} {1.8mm} {
\begin{tabular}{l|lllllll}
\toprule
\multicolumn{2}{c}{$m\times n$}  &$1000\times 100$  &$2000\times 200$  &$3000\times 300$   &$4000\times 400$   &$5000\times 500$   &$6000\times 600$\\
\midrule
\multirow{2}{*}{RP}
&IT        &35623.5       &82113.7      &118555.5     &157981.9      &197158.3      &237194.7        \\
&CPU       &2.2570e-01    &9.2740e-01   &1.9676e+00   &4.2829e+00    &7.1908e+00    &9.8938e+00  \\
\midrule
\multirow{2}{*}{SKM}
&IT        &1372.8        &3211.9       &4920.9       &7060.8        &8968.7        &10605.7       \\
&CPU       &1.4610e-01    &5.9440e-01   &1.2913e+00   &3.0899e+00    &5.1660e+00    &7.6251e+00    \\
\midrule
\multirow{2}{*}{GSKM}
&IT        &1191.4        &2538.5       &4017.2       &5634          &6852.4        &8389.6        \\
&CPU       &1.2800e-01    &5.1520e-01   &1.1971e+00   &3.2318e+00    &6.9181e+00    &1.2251e+01   \\
\midrule
\multirow{2}{*}{PASKM}
&IT        &235.8         &554.5        &875.8        &1257.4        &1563.5        &2026.9       \\
&CPU       &2.9000e-02    &1.0850e-01   &2.6190e-01   &7.4220e-01    &1.6022e+00    &3.0270e+00  \\
\midrule
\multirow{1}{*}{GRABP-c}
&IT        &1689.3        &1908.5      &2030.5        &2023.2        &2099.9        &2041.6    \\
\multirow{1}{*}{$(\alpha=\frac{1}{\zeta})$}
&CPU       &2.9940e-01    &3.8480e-01   &6.0820e-01   &1.2052e+00    &1.9307e+00    &2.8221e+00   \\
\midrule
\multirow{1}{*}{GRABP-c}
&IT        &791.9         &921.3        &925.4        &943.4         &943.4         &945.4   \\
\multirow{1}{*}{$(\alpha=\frac{1.95}{\zeta})$}
&CPU       &1.5160e-01    &2.1930e-01   &3.4970e-01   &7.0140e-01    &1.0905e+00    &1.6575e+00   \\
\midrule
\multirow{1}{*}{GRABP-a}
&IT        &175.8         &205.9        &231.7        &217.7         &220.8         &221.6     \\
\multirow{1}{*}{$(w=1)$}
&CPU       &3.2000e-02    &7.2200e-02   &1.3420e-01   &2.6490e-01    &4.2010e-01    &6.8350e-01   \\
\midrule
\multirow{1}{*}{GRABP-a}
 &IT    &\bf{44.2}       &\bf{58.7}      &\bf{67.8}      &\bf{72.9}      &\bf{77.5}      &\bf{83.8} \\
\multirow{1}{*}{$(w=1.95)$}
&CPU    &\bf{1.2000e-02} &\bf{4.3000e-02}&\bf{9.4000e-02}&\bf{2.0200e-01}&\bf{3.1530e-01}&\bf{5.3490e-01}      \\
\bottomrule
\end{tabular}}
\end{table*}

\begin{table*}[htp]
\footnotesize
\centering
\caption{ \title{\small Numerical results for $m$-by-$n$ random sparse  matrices $A$.}}
\label{Tabsparse}
\renewcommand\arraystretch{1.1}
\setlength{\tabcolsep} {1.8mm} {
\begin{tabular}{l|lllllll}
\toprule
\multicolumn{2}{c}{$m\times n$}  &$1000\times 100$  &$2000\times 200$  &$3000\times 300$   &$4000\times 400$   &$5000\times 500$   &$6000\times 600$\\
\midrule
\multirow{2}{*}{RP}
&IT        &59492.3       &138220.8     &196983.3     &204329.3      &253219.4      &314986.3        \\
&CPU       &3.7260e-01    &1.5622e+00   &3.2701e+00   &5.5005e+00    &9.1793e+00    &1.3106e+01  \\
\midrule
\multirow{2}{*}{SKM}
&IT        &1156.4        &2582.7       &4748.9       &5725.9        &8232.5        &8557.4     \\
&CPU       &1.1860e-01    &4.6110e-01   &1.2782e+00   &2.4831e+00    &4.7977e+00    &6.2059e+00    \\
\midrule
\multirow{2}{*}{GSKM}
&IT        &1011.7        &2136.8       &3487.2       &4466.7        &6439.7        &6935.7       \\
&CPU       &1.1200e-01    &4.2200e-01   &1.0611e+00   &2.5762e+00    &6.4646e+00    &1.0184e+01   \\
\midrule
\multirow{2}{*}{PASKM}
&IT        &608.5         &1314.2       &2037.4       &2621.5        &3868          &4289.6       \\
&CPU       &6.6000e-02    &2.6760e-01   &6.2570e-01   &1.5164e+00    &3.8957e+00    &6.3062e+00  \\
\midrule
\multirow{1}{*}{GRABP-c}
&IT        &4114.1        &3662.9       &3505.8       &3426.3        &3480.7        &3086.7    \\
\multirow{1}{*}{$(\alpha=\frac{1}{\zeta})$}
&CPU       &6.4530e-01    &7.4430e-01   &9.7310e-01   &1.9008e+00    &2.8947e+00    &3.8220e+00   \\
\midrule
\multirow{1}{*}{GRABP-c}
&IT        &2078.7        &1788.3       &1648.4       &1705.3        &1666.3        &1481.9  \\
\multirow{1}{*}{$(\alpha=\frac{1.95}{\zeta})$}
&CPU       &3.3500e-01    &3.9260e-01   &5.3050e-01   &1.0581e+00    &1.6411e+00    &2.1893e+00   \\
\midrule
\multirow{1}{*}{GRABP-a}
&IT        &107.4         &141.9        &143.9        &150.5         &167.2         &156.8     \\
\multirow{1}{*}{$(w=1)$}
&CPU       &2.1800e-02    &5.4000e-02   &1.1300e-01   &2.3410e-01    &3.9450e-01    &6.0120e-01  \\
\midrule
\multirow{1}{*}{GRABP-a}
&IT    &\bf{39.1}      &\bf{51.7}      &\bf{54.4}      &\bf{56.2}      &\bf{57.4}      &\bf{58.8} \\
\multirow{1}{*}{$(w=1.95)$}
&CPU   &\bf{1.1000e-02}&\bf{4.4000e-02}&\bf{9.3500e-02}&\bf{1.9000e-01}&\bf{3.2210e-01}&\bf{4.9030e-01}      \\
\bottomrule
\end{tabular}}
\end{table*}
\subsection{Experiments on real-world test instances }

 In this subsection, we consider the following two types of real-world test instances: the SuiteSparse Matrix Collection   and the sparse Netlib LP instances.
 \subsubsection{The SuiteSparse Matrix Collection }

In Table \ref{TabSuiteSparse}, the  coefficient matrix $A$ is chosen from the SuiteSparse Matrix Collection.
For details, we list their sizes, densities, condition numbers (i.e., cond($A$)), and squared Euclidean norms in Table \ref{Tabmatrices}, where  the density of a matrix is defined by
$$ \text{density} =\frac{ \text{number of nonzeros of an $m$-by-$n$ matrix}}{mn}.$$
In addition, the right-hand side $b$ is generated randomly as $b=0.5Ax_1+0.5Ax_2+x_3$.  From Table \ref{TabSuiteSparse}, we can also observe that the GRABP, PASKM, GSKM, and SKM methods outperform the RP method  in terms of both the iteration count and the CPU time. In addition, GRABP-a with $w=1.95$ is more efficient than other methods.
\begin{table*}[htp]
\footnotesize
\centering
\caption{ \title{\small Properties of $m$-by-$n$   matrices $A$  from the SuiteSparse  Matrix Collection.}}
\label{Tabmatrices}
\renewcommand\arraystretch{1.1}
\setlength{\tabcolsep} {5.5mm} {
\begin{tabular}{lllllll}
\toprule
\text{Name}  &$m$  &$n$  &\text{density}   &\text{cond($A$)}   &$\|A\|^2_2$\\
\midrule
\text{ash958}           &958        &292        &0.68$\%$           &3.2014           &17.9630                 \\
\text{illc1033}         &1033       &320        &1.43$\%$           &1.8888e+04       &4.5983     \\
\text{well1033}        &1033       &320        &1.43$\%$           &166.1333         &3.2635     \\
\text{ch\_8\_b1}        &1568       &64         &3.13$\%$           &3.3502e+15       &56     \\
\text{illc1850}         &1850       &712        &0.66$\%$           &1.404e+03        &4.5086     \\
\text{Franzl}           &2240       &768        &0.30$\%$           &8.0481e+15       &17.4641     \\
\bottomrule
\end{tabular}}
\end{table*}

\begin{table*}[htp]
\footnotesize
\centering
\caption{\title{\small  Numerical results for  matrices $A$ from the  SuiteSparse   Matrix Collection.}}
\label{TabSuiteSparse}
\renewcommand\arraystretch{1.0}
\setlength{\tabcolsep} {1.8mm} {
\begin{tabular}{l|lllllll}
\toprule
\multicolumn{2}{c}{Name}  &\text{ash958}   &\text{illc1033}  &\text{well1033}   &\text{ch\_8\_b1}   &\text{illc1850} &\text{Franzl}\\
\midrule
\multirow{2}{*}{RP}
&IT           &4751          &5353.6        &11629.9       &9139.5        &16288.9       &91585.7     \\
&CPU          &3.2700e-02    &3.8000e-02    &7.6000e-02    &8.4000e-02    &2.2460e-01    &1.5296e+00 \\
\midrule
\multirow{2}{*}{SKM}
&IT           &584.4         &271.4         &451.8         &429.5         &588.1         &4718.7        \\
&CPU          &8.9200e-02    &4.8500e-02    &7.7100e-02    &3.4300e-02    &3.0570e-01    &2.7192e+00    \\
\midrule
\multirow{2}{*}{GSKM}
&IT           &430.3         &193.6         &320.9         &366.8         &427.5         &3553     \\
&CPU          &7.6000e-02    &3.6100e-02    &5.8600e-02    &4.2200e-02    &2.4420e-01    &2.7333e+00  \\
\midrule
\multirow{2}{*}{PASKM}
&IT           &505.4         &265.2         &455.2         &426.1         &605.8         &4696     \\
&CPU          &8.3300e-02    &5.2400e-02    &8.3600e-02    &5.0500e-02    &3.8750e-01    &3.6121e+00 \\
\midrule
\multirow{1}{*}{GRABP-c}
&IT           &594.7         &919.5         &769.6         &1002.6        &1013.2        &1711.8     \\
\multirow{1}{*}{$(\alpha=\frac{1}{\zeta})$}
&CPU          &1.4940e-01    &1.4340e-01    &1.2730e-01    &3.9360e-01    &4.6850e-01    &9.3570e-01   \\
\midrule
\multirow{1}{*}{GRABP-c}
&IT           &193.4         &421.8         &409.3         &375.3          &487.2        &704.7    \\
\multirow{1}{*}{$(\alpha=\frac{1.95}{\zeta})$}
&CPU          &4.8600e-02    &7.7500e-02    &7.9500e-02    &1.5350e-01     &2.7360e-01   &4.5560e-01    \\
\midrule
\multirow{1}{*}{GRABP-a}
&IT           &28.3          &16.6          &17            &43.4           &21.8         &262.9      \\
\multirow{1}{*}{$(w=1)$}
&CPU          &1.3500e-02    &1.3500e-02    &1.2000e-02    &2.8500e-02     &5.5000e-02   &1.9500e-01   \\
\midrule
\multirow{1}{*}{GRABP-a}
 &IT          &\bf{17.7}     &\bf{8.5}      &\bf{9.6}      &\bf{22.9}     &\bf{11.5}     &\bf{48.5} \\
\multirow{1}{*}{$(w=1.95)$}
&CPU       &\bf{1.0300e-02}&\bf{1.2400e-02}&\bf{1.1500e-02}&\bf{2.1900e-02}&\bf{4.8100e-02}&\bf{9.8000e-02}      \\
\bottomrule
\end{tabular}}
\end{table*}
\subsubsection{Netlib LP instances }
In this subsection, we compare the performance of the algorithms for solving Netlib LP test instances. We follow the standard framework used by De Loera et al.\cite{De17} and Morshed et al.\cite{morshed2021sampling} in their work for linear feasibility problems. The problem instances are transformed form standard LP problems(i.e., min $c^\top x$ subject to $Ax=b$, $l\leq x\leq u$ with optimum value $p^*$ ) to an equivalent linear feasibility formulation (i.e., ${\bf A}x\leq {\bf b}$, where ${\bf A}=[A^\top \: -A^\top \:  I \: -I \: c]^\top$ and ${\bf b}=[ b^\top \:  -b^\top \:  u^\top \: -l^\top \: p^*]^\top$).

In Table \ref{realword}, we test a total of 5 instances.  The tolerance gap of the algorithm is $\phi=10^{-2}$ when testing the instances with corner labels $*$.  The tolerance gap of the algorithm is $\phi=10^{-3}$  when testing the instances without corner label.  From Table \ref{realword}, we know that Algorithm GRABP-a takes less computing time compared to the other algorithms.

\begin{table*}[htp]
\footnotesize
\centering
\caption{\title{\small  Numerical results for sparse Netlib LP instances.}}
\label{realword}
\renewcommand\arraystretch{1.0}
\setlength{\tabcolsep} {2mm} {
\begin{tabular}{l|lllllll}

			\toprule
			\multicolumn{2}{c}{Instance}  &share2b        &recipe          &scsd1*         &scsd6*            &fit1d    \\
			\multicolumn{2}{c}{Dimensions}&$189\times79$  &$434\times180$  &$915\times760$ &$1645\times1350$ &$2078\times1026$ \\
			\midrule
			\multirow{2}{*}{RP}
			&IT        &4089705.5    &2128891.2   &145623       &79436.2       &31113.9       \\
			&CPU       &5.0001e+01   &5.0000e+01  &1.4228e+01   &5.0002e+01    &1.8505e+01     \\
			\midrule
			\multirow{2}{*}{SKM}
			&IT        &290975.8     &4225.8      &5244.7       &38117.9       &40846         \\
			&CPU       &1.0075e+01   &2.9540e-01  &1.4213e+00   &5.0003e+01    &4.7280e+01      \\
			\midrule
			\multirow{2}{*}{GSKM}
			&IT        &239814.8     &3930.5      &5249.3       &29043.5      &26364.7        \\
			&CPU       &9.0729e+00   &2.9530e-01  &1.6096e+00   &5.0001e+01    &3.9265e+01     \\
			\midrule
			\multirow{2}{*}{PASKM}
			&IT        &228577.9     &4014.9      &5447.5       &29059.2      &32565.9      \\
			&CPU       &8.5962e+00   &3.2110e-01  &1.6945e+00   &5.0002e+01    &4.8777e+01     \\
			\midrule
			\multirow{1}{*}{GRABP-c}
			&IT        &516628       &419839.1    &97486.8      &38620.3      &28525.7       \\
			\multirow{1}{*}{$(\alpha\zeta =1)$}
			&CPU       &5.0001e+01   &5.0000e+01  &2.5225e+01   &5.0004e+01    &3.5774e+01     \\
			\midrule
			\multirow{1}{*}{GRABP-c}
			&IT        &505367       &419313.9    &48003.7      &38448.5       &14693.1      \\
			\multirow{1}{*}{$(\alpha\zeta=1.95)$}
			&CPU       &5.0000e+01   &5.0000e+01  &1.2406e+01   &5.0004e+01    &1.8724e+01     \\
			\midrule
			\multirow{1}{*}{GRABP-a}
			&IT        &165189.8     &568.6       &1047.4       &20749.2        &29.3          \\
			\multirow{1}{*}{$(w=1)$}
			&CPU  &1.5914e+01 &\bf{6.7000e-02}&\bf{2.8580e-01}&\bf{2.6823e+01}    &1.0850e-01  \\
			\midrule
			\multirow{1}{*}{GRABP-a}
			&IT    &53832.5        &2763.4      &51721        &38481.2      &10.9  \\
			\multirow{1}{*}{$(w=1.95)$}
			&CPU  &\bf{5.2471e+00} &3.1490e-01 &1.3156e+01    &5.0004e+01 &\bf{8.3900e-02}   \\
			\bottomrule
\end{tabular}}
\end{table*}

\subsection{Remarks about the choice of block}
Finally, we show the performance of the GRABP-a algorithm under different blocks.
In Figure \ref{fig:2}, we  test the performance of GRABP-a with $w=1$ and GRABP-a with $w=1.95$ for different number of blocks, where  $t=2$, $t=5$, $t=10$, $t=50$, $t= 100 $, $t=\lceil\| A\|^2_2\rceil=305$, and
the coefficient matrix   is a randomly generated $5000$-by-$500$ sparse matrix.
In Figure \ref{fig:1}, the coefficient matrix   is a randomly generated $5000$-by-$500$ dense matrix.
Since $\lceil\| A\|^2_2\rceil=8492$ exceeds the number of rows of the matrix in such case, we do not test the case where $t=\lceil\| A\|^2_2\rceil$, and we consider the case where  $t=2$, $t=5$, $t=10$, $t=50$, $t= 100 $, and $t=200$.

From Figure  \ref{fig:2} and Figure  \ref{fig:1}, we can see that the convergence rate of the GRABP-a algorithm slows down as the number of blocks $t$ increases.
One can also see that the smaller the number of blocks $t$, the fewer iterative steps the GRABP-a algorithm requires. In particular, the algorithm performs best when $t=2$. In addition, the GRABP-a algorithm with $w=1.95$ performs better than the GRABP-a algorithm with $w=1$ during all of the tests.

\begin{figure}[h]
	\centering
	\begin {minipage}[ht]{0.496\linewidth}
	\centering
	\includegraphics[width=1\linewidth]{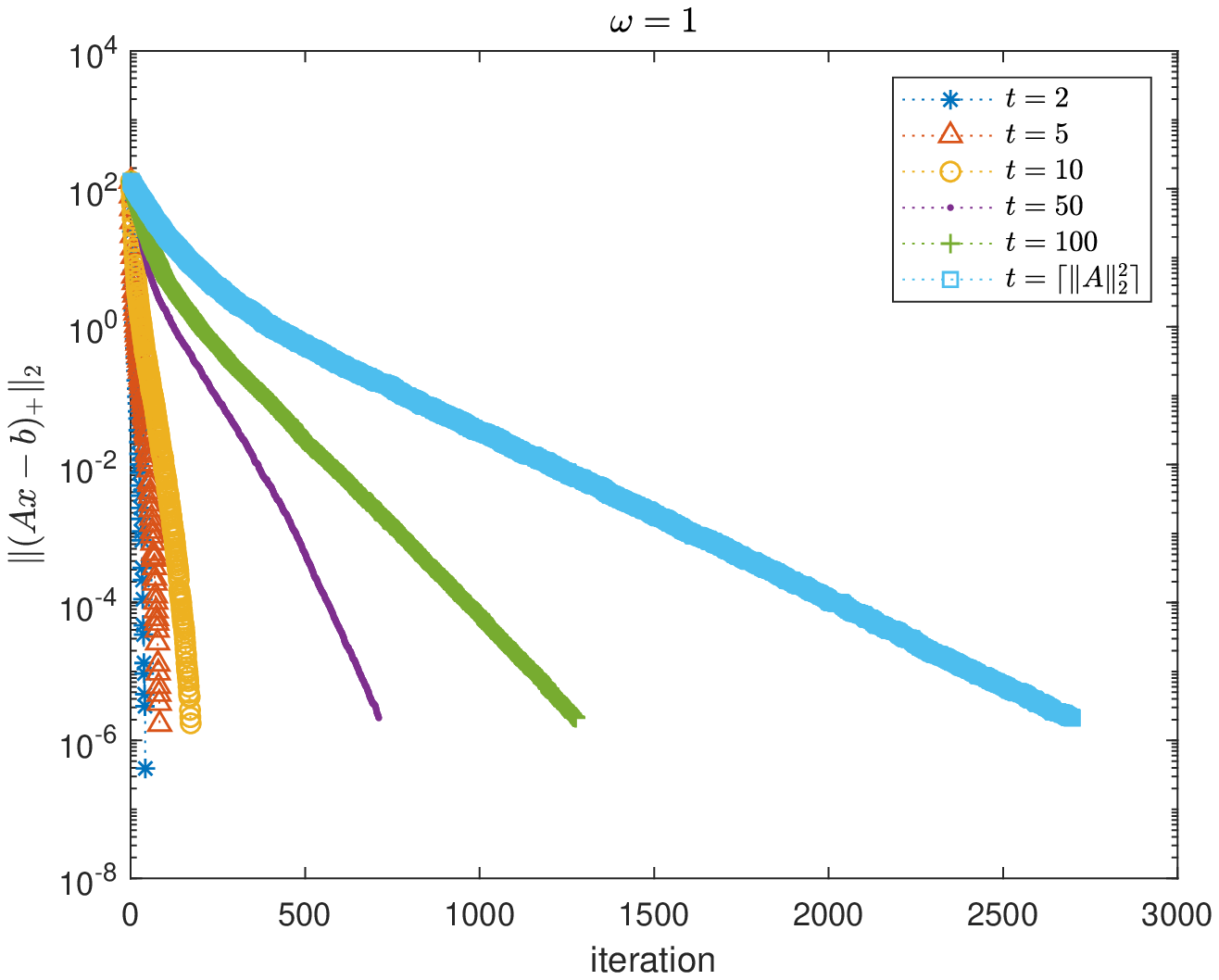}
\end{minipage}
\begin {minipage}[ht]{0.496\linewidth}
\centering
\includegraphics[width=1\linewidth]{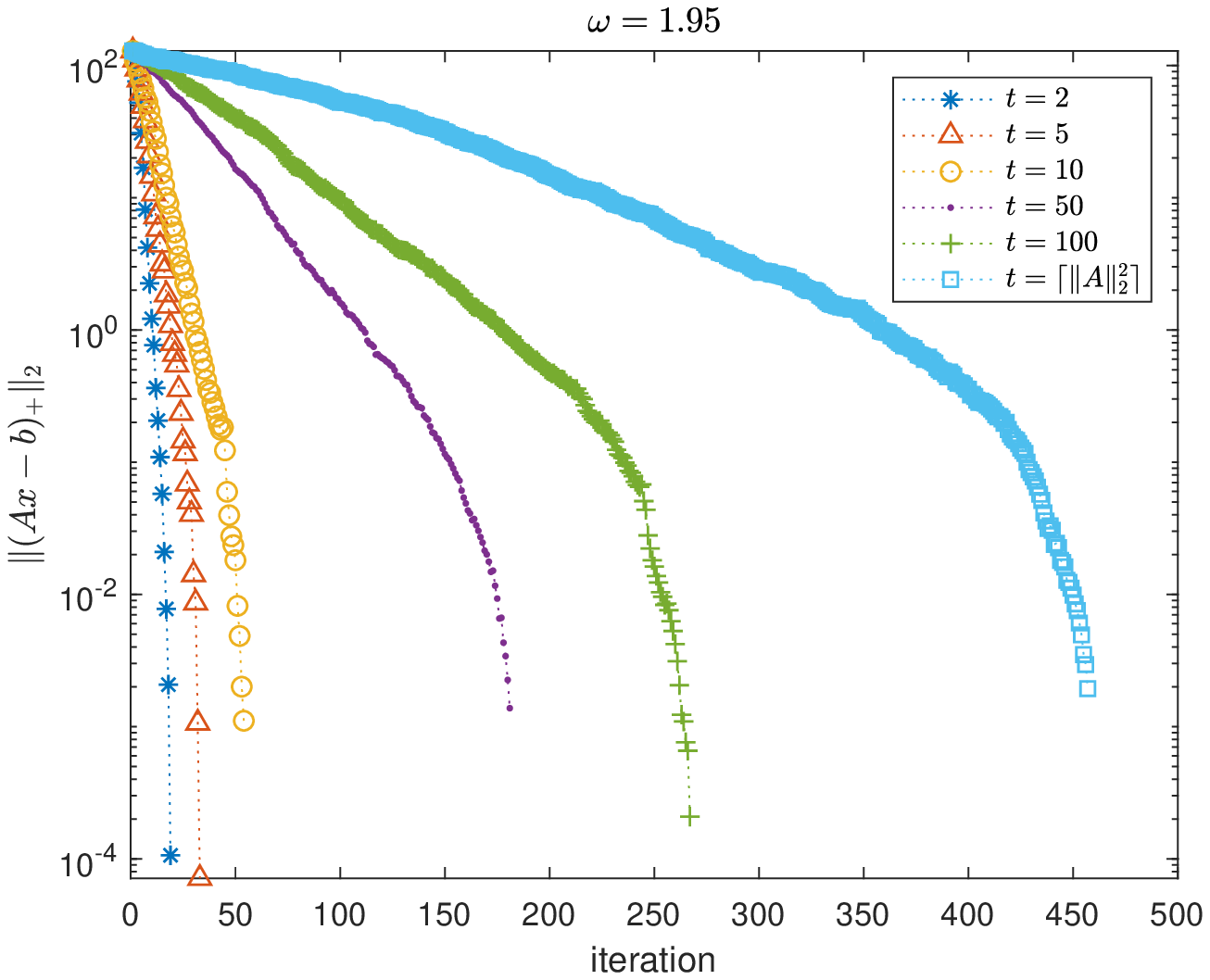}
\end{minipage}
\caption{GRABP-a: A $5000$-by-$500$ randomized sparse  matrix $A$.}\label{fig:2}
\end{figure}

\begin{figure}[t]
\begin {minipage}[t]{0.496\linewidth}
\centering
\includegraphics[width=1\linewidth]{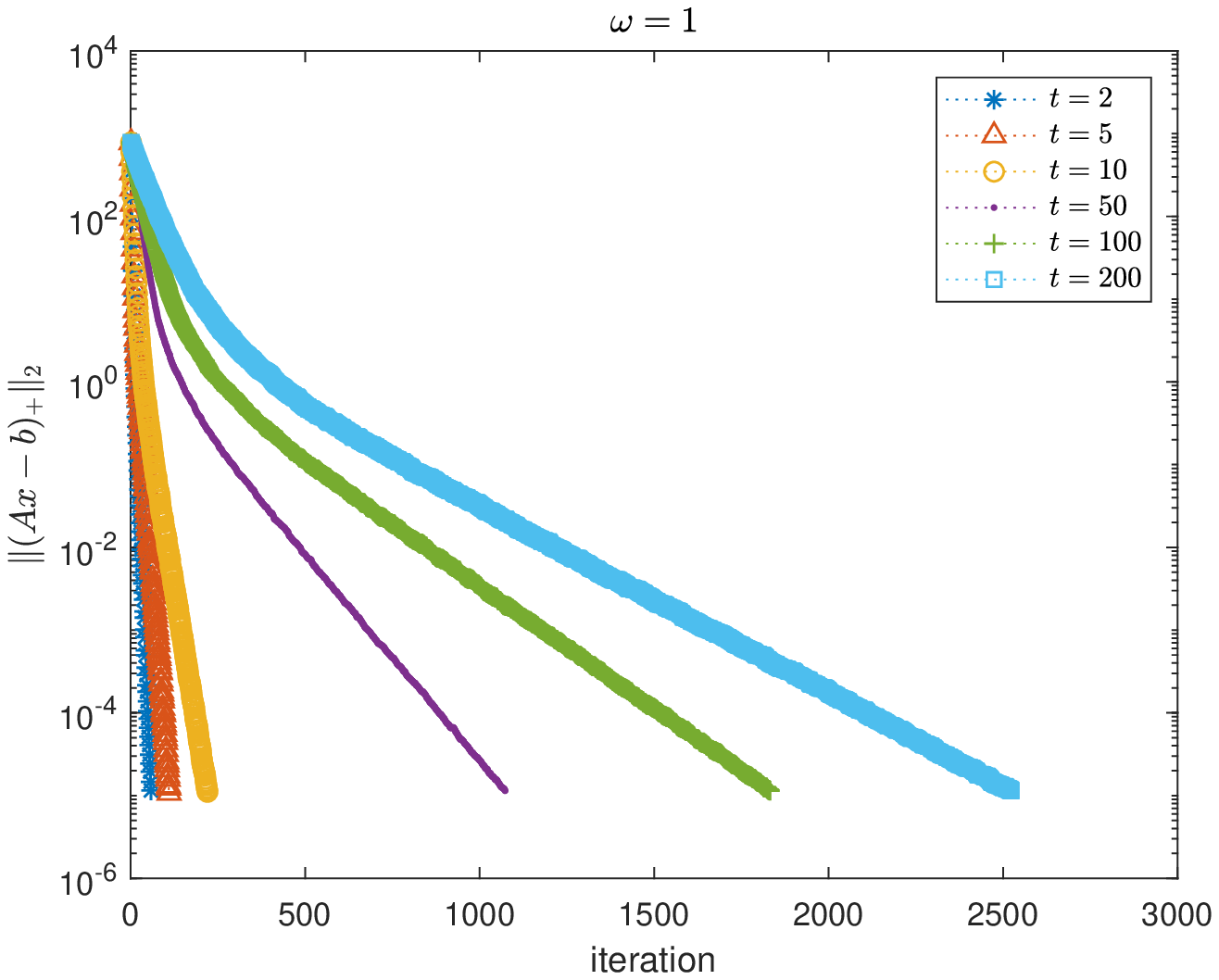}
\end{minipage}
\begin {minipage}[t]{0.496\linewidth}
\centering
\includegraphics[width=1\linewidth]{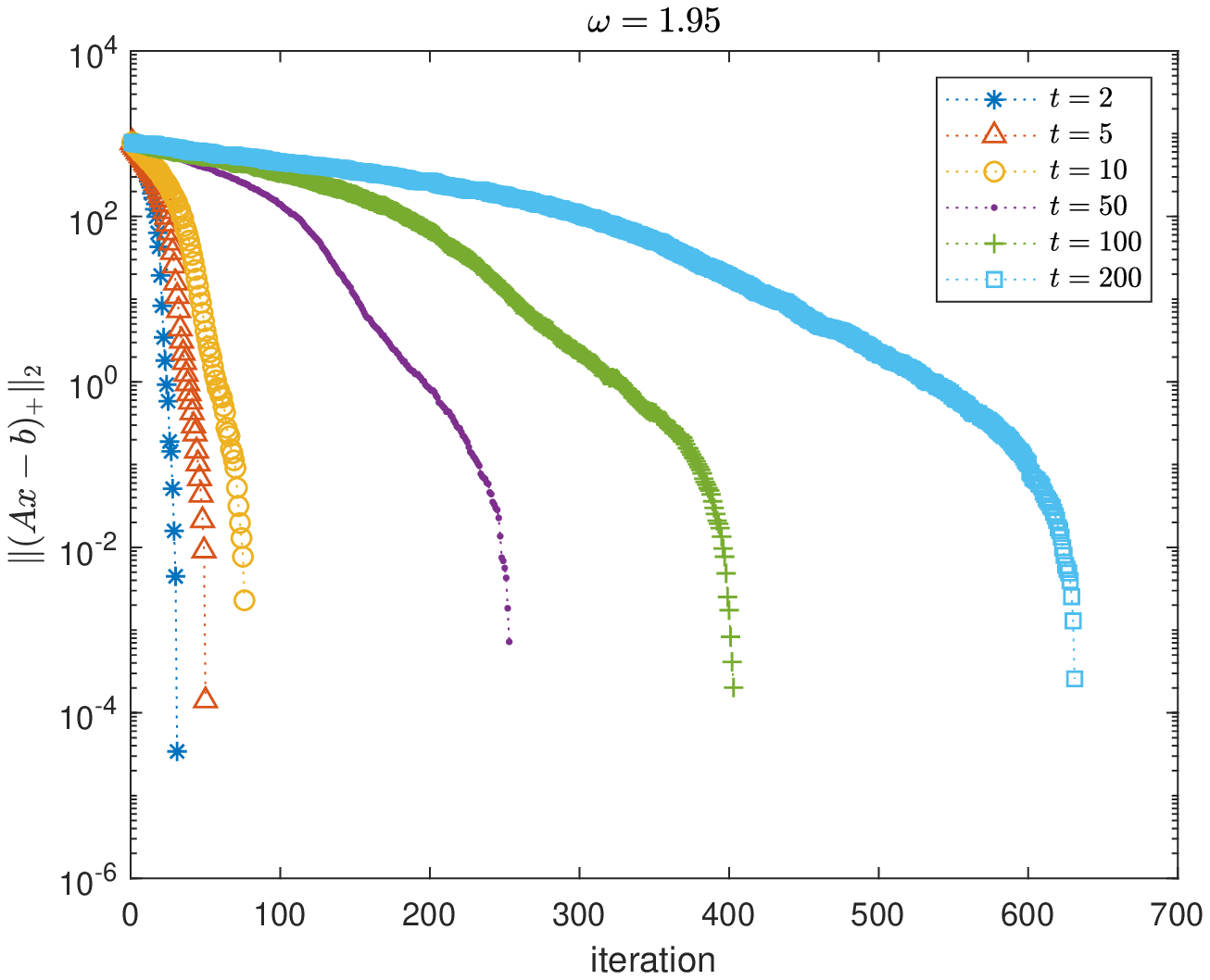}
\end{minipage}
\caption{GRABP-a: A $5000$-by-$500$ randomized dense matrix $A$.}\label{fig:1}
\end{figure}

\section{Conclusions}\label{sec5}
This paper introduced the GRABP method, which inherits ideas from the greedy probability criterion and the average block method, for solving linear feasibility problems.
It was proved that the GRABP method converges linearly in exception with two kinds of choices of extrapolation steps.
Numerical results show that the GRABP method works better over several state-of-the-art methods.

Finally, it should be pointed out that the RP method and its variants only ensure that the iteration sequence $\{x^k\}$ converges to a certain feasible point in $S$. Precisely, Theorems \ref{convergence_1} and \ref{convergence_2} only guarantee that the distance between  $x^k$ and $S$ converges to zero.
However, in practice one may want to find solutions with certain structures in $S$, for example, the least norm solution. The Hildreth's method \cite{hildreth1957quadratic,iusem1990convergence,lent1980extensions,jamil2015hildreth} is also a row action method for solving linear feasibility problems, but with one more benefit: finding the closest point in the solution set to a given point $x^0$, i.e. its iteration sequence $\{x^k\}$ converges to $P_{S}(x^0)$.
This topic is practically valuable and theoretically meaningful, and will be
investigated in detail and discussed in depth in the future.

\nocite{*}
\bibliography{grabp}

\end{document}